\documentclass[12pt,reqno]{amsart}
\usepackage[bookmarks=false]{hyperref} 
\usepackage{amssymb}
\usepackage{amsmath}
\usepackage{amsthm} 
\usepackage{amsfonts}
\usepackage{bbm}
\usepackage{mathtools}
\usepackage{color} 
\usepackage[svgnames]{xcolor}
\usepackage[utf8]{inputenc}
\usepackage[T1]{fontenc}
\usepackage{graphicx}
\usepackage{placeins}
\usepackage{setspace}

\usepackage[margin=1in]{geometry}
\usepackage{times}

\makeatletter
\newcommand{\mainsectionstyle}{%
  \renewcommand{\@secnumfont}{\bfseries}
  \renewcommand\section{\@startsection{section}{2}%
    \z@{.5\linespacing\@plus.7\linespacing}{-.5em}%
    {\normalfont\bfseries}}%
}
\makeatother

\hypersetup{
  colorlinks=true,
  linkcolor=Blue  ,          
  citecolor=Red,        
  filecolor=Magenta,      
  urlcolor=Green,           
pdfpagemode=UseOutlines,
pdftitle={},
pdfauthor={Tin Van Phan <van-tin.phan@univ-tlse3.fr>},
pdfsubject={Infinite soliton for derivative nonlinear Schr\"odinger equation},
pdfkeywords={soliton, Schr\"oginer equation} 
}

\numberwithin{equation}{section}
\numberwithin{figure}{section}

\newtheorem{theorem}{Theorem}[section]
\newtheorem{proposition}[theorem]{Proposition}
\newtheorem{lemma}[theorem]{Lemma}

\theoremstyle{definition}
\newtheorem{definition}[theorem]{Definition}

\theoremstyle{remark}
\newtheorem{remark}[theorem]{Remark}

\DeclarePairedDelimiter{\norm}{\lVert}{\rVert}

\newcommand{\RNum}[1]{\uppercase\expandafter{\romannumeral #1\relax}}

\newcommand{\R}{\mathbb{R}}

\renewcommand{\leq}{\leqslant}
\renewcommand{\geq}{\geqslant}

\DeclareMathAlphabet{\mathpzc}{OT1}{pzc}{m}{it}
\renewcommand{\Re}{\mathcal R\!\mathpzc{e}}
\renewcommand{\Im}{\mathcal I\!\mathpzc{m}}

\begin{document}

\mainsectionstyle
\doublespacing

\title[Algebraic standing waves]{Instability of algebraic standing waves for nonlinear Schr\"odinger equations with triple power nonlinearities}

\author[Phan Van Tin]{Phan Van Tin}

\address[Phan Van Tin]{Institut de Math\'ematiques de Toulouse ; UMR5219,
  \newline\indent
  Universit\'e de Toulouse ; CNRS,
  \newline\indent
  UPS IMT, F-31062 Toulouse Cedex 9,
  \newline\indent
  France}
\email[Phan Van Tin]{van-tin.phan@univ-tlse3.fr}


\date{\today}

\begin{abstract}
We consider the following triple power nonlinear Schr\"odinger equation:
\begin{equation*}
iu_t+\Delta u+a_1|u|u+a_2|u|^{2}u+a_3|u|^{3}u=0.
\end{equation*}
We are interested in \textit{algebraic standing waves} i.e standing waves with algebraic decay above equation in dimensions $n$ ($n=1,2,3$). We prove the instability of these solutions in the cases DDF (we use abbreviation D: defocusing ($a_i<0$), F: focusing ($a_i>0$)) and DFF when $n=2,3$ and in the case DFF with $a_1=-1$, $a_3=1$ and $a_2<\frac{32}{15\sqrt{6}}$ when $n=1$. Under these assumptions, the standing waves are orbitally unstable in the case of small positive frequency. When the highest power is $L^2(\R^n)$-supercritical power (for $n=2,3$), $a_1=-1$, $a_3=1$ and $a_2>-\varepsilon$ for $\varepsilon>0$ small enough in the case $n=3$, we prove that standing waves with positive frequency are unstable by blow up.
\end{abstract}

\maketitle
\tableofcontents

\section{Introduction}
In this paper, we are interested in the following triple power nonlinear Schr\"odinger equation:
\begin{equation}\label{triple power equation}
iu_t+\Delta u+a_1|u|u+a_2|u|^{2}u+a_3|u|^{3}u=0, \quad (t,x) \in \R \times \R^n,
\end{equation}
where $a_1,a_2,a_3\in\R$ and $n \in \{1,2,3\}$. 

In the cases $n=1,2,3$, \eqref{triple power equation} is in $H^1(\R^n)$-subcritical case. This ensures that \eqref{triple power equation} is locally well posed in $H^1(\R^n)$ (see e.g \cite{Ca03}). The standing waves of \eqref{triple power equation} are solutions of the form $u_{\omega}(t,x)=e^{i\omega t}\phi_{\omega}(x)$, where $\phi_{\omega}$ solves:
\begin{equation}
\label{SW profile}
-\omega \phi_{\omega} +\Delta\phi_{\omega}+a_1|\phi_{\omega}|\phi_{\omega}+a_2|\phi_{\omega}|^{2}\phi_{\omega}+a_3|\phi_{\omega}|^{3}\phi_{\omega}=0.
\end{equation}

Consider the focusing nonlinear Schr\"odinger equation with single power $|u|^{p-1}u$. In this case, the standing waves are orbitally stable if $p<1+\frac{4}{n}$ ($L^2(\R^n)$-subcritical) and orbitally unstable if $p> 1+\frac{4}{n}$ ($L^2(\R^n)$-supercritical). In this paper, we study the stability and instability of standing waves with multiple power nonlinearity combining $L^2(\R^n)$-subcritical power and $L^2(\R^n)$-supercritical power (for $n=2,3$) and all $L^2(\R^n)$-subcritical powers (for $n=1$).

In \cite{LiTsIa21}, the authors study existence and stability of standing waves of \eqref{triple power equation} in one dimension. Existence of standing waves is obtained by ODE arguments. By studying the properties of the nonlinearity, the authors give domains of parameters for existence and nonexistence of standing waves. Stability results are obtained by studying the sign of an integral found by Iliev and Kirchev \cite{IlKi93}, based on the criteria of stability of Grillakis, Shatah and Strauss \cite{GrShSt87,GrShSt90,ShSt85}. Using this criteria, in \cite{Ma08}, the author proved the stability and the instability of standing waves for $1$-dimensional nonlinear Schr\"odinger equation with double power and triple power nonlinearity. In the case of triple power nonlinearity, the author showed that stability of standing waves change by $\omega$, two and three times. This does not occur in the cases of single power and double power.

In the special case $\omega=0$, the profile $\phi_0$, which for convenience we denote by $\phi$, satisfies:
\begin{equation}
\label{SW algebraic profile}
\Delta\phi+a_1|\phi|\phi+a_2|\phi|^{2}\phi+a_3|\phi|^{3}\phi=0.
\end{equation}  

The equation \eqref{SW algebraic profile} can be rewritten as $S'(\phi)=0$ where $S$ is defined by
\begin{equation}
S(v):= \frac{1}{2}\norm{\nabla v}_{L^2}^2-\frac{a_1}{3}\norm{v}^{3}_{L^{3}}-\frac{a_2}{4}\norm{v}^{4}_{L^{4}}-\frac{a_3}{5}\norm{v}^{5}_{L^{5}}. \label{define of S omega}
\end{equation}
Define
\begin{align}
X&:=\dot{H}^1(\R^n) \cap L^3(\R^n), \quad \text{ and } \norm{u}_X:= \norm{\nabla u}_{L^2}+\norm{u}_{L^3},
\label{defind of X}\\
d&:=\inf\{S(v): v \in X\setminus \{0\}, S'(v)=0\}. \label{define of d}
\end{align}

The algebraic standing waves are standing waves with algebraic decay. In this paper, we are only interested in a special kind of algebraic standing waves which are minimizers of the problem \eqref{define of d}. Throughout this paper, for convenience, we define an algebraic standing wave as a solution of \eqref{SW algebraic profile} solving problem \eqref{define of d}. Thus, the function $\phi$ is an algebraic standing wave of \eqref{triple power equation} if $\phi\in\mathcal{G}$, where $\mathcal{G}$ is defined by
\begin{equation}
\mathcal{G}:= \{v\in X\setminus\{0\}:S'(v)=0,S(v)=d\}.\label{define of G}
\end{equation}

The instability of algebraic standing waves was studied in \cite{FuHa21} for double power nonlinearities. Using similar arguments as in \cite{FuHa21}, we study existence and instability of algebraic standing waves for the nonlinear Schr\"odinger equation with triple power nonlinearities \eqref{triple power equation}.    

First, we study the existence of algebraic standing waves of \eqref{triple power equation}. As in \cite{LiTsIa21}, we will use the abbreviation D: defocusing when $a_i<0$ and F: focusing when $a_i>0$. In Section \ref{sec1}, we prove the following result.
\begin{proposition}
\label{pro1}
Let $n=1$. The equation \eqref{SW algebraic profile} has a unique even positive solution $\phi$ in the space $H^1(\R)$ in the following cases: DFF, DDF, DFD and $a_1=a_3=-1$, $a_2>\frac{8}{\sqrt{15}}$. Moreover, all solutions of \eqref{SW algebraic profile} are of the form $e^{i\theta}\phi(x-x_0)$ for some $\theta,x_0\in\R$. They are all algebraic standing waves of \eqref{triple power equation}. 
\end{proposition}

In high dimensions, the situation is more complex than in the one dimension. The solutions of \eqref{SW algebraic profile} are very diverse. It is not easy to describe all such solutions as in the dimension one. Thus, classifying the algebraic standing waves of \eqref{triple power equation} is not easy problem. It turns out that a radial positive solutions of \eqref{SW algebraic profile} is also an algebraic standing wave of \eqref{triple power equation}. To study the positive radial solutions of \eqref{SW algebraic profile}, we prove the following result in Section \ref{sec1}.  
\begin{proposition}
\label{pro2}
Let $n=2,3$ and DDF or DFF. Then there exists a unique radial positive solution of \eqref{SW algebraic profile}.
\end{proposition}

Before stating the next results, we need some definitions. Firstly, we define the \textit{Nehari functional} as follows:
\begin{equation}\label{define of K omega}
K(v):= \left<S'(v),v\right>=\norm{\nabla v}^2_{L^2}-a_1\norm{v}^{3}_{L^3}-a_2\norm{v}^{4}_{L^{4}}-a_3\norm{v}^{5}_{L^{5}}.
\end{equation}
The \textit{rescaled function} is defined by:
\begin{equation}\label{define of v^lambda}
v^{\lambda}(x):= \lambda^{\frac{n}{2}}v(\lambda x).
\end{equation}
The following is \textit{Pohozhaev functional}:
\begin{equation}\label{define of P}
P(v):= \partial_{\lambda}S(v^{\lambda})\vert_{\lambda=1}=\norm{\nabla v}^2_{L^2}-\frac{n a_1}{6}\norm{v}^3_{L^3}-\frac{n a_2}{4}\norm{v}^4_{L^4}-\frac{3n a_3}{10}\norm{v}^{5}_{L^{5}}.
\end{equation}
The \textit{Nehari manifold} is defined by:
\begin{equation*}
\mathcal{K}:= \{v\in X\setminus\{0\}: K(v)=0\}.
\end{equation*}
Moreover, we consider the following minimization problem:
\begin{equation}\label{define of mu}
\mu:= \inf\left\{S(v): v\in \mathcal{K}\right\}.
\end{equation}
The following is the set of minimizers of problem \eqref{define of mu}:
\begin{equation}\label{define of M}
\mathcal{M}:=\{v \in \mathcal{K}: S(v)=\mu\}.
\end{equation}
Finally, we define a specific set which uses in our proof:
\begin{equation}\label{define of B omega}
\mathcal{B}:= \left\{ v \in H^1(\R^n): S(v)<\mu, P(v)<0\right\}.
\end{equation}

It turns out that the solution of \eqref{SW algebraic profile} given by Proposition \ref{pro2} satisfies a variational characterization and each algebraic standing wave of \eqref{triple power equation} is up to phase shift and translation of this special solution. More precise, in Section \ref{sec2}, we prove the following result. 

\begin{proposition} \label{pro variational characteristic}
Let $n=1,2,3$ and DDF or DFF. Then the radial positive solution $\phi$ of \eqref{SW algebraic profile} given by Proposition \ref{pro1} and Proposition \ref{pro2} satisfies
\[
S(\phi)=\mu.
\] 
where $S$ and $\mu$ are defined as in \eqref{define of S omega}, \eqref{define of mu} respectively. Moreover, all algebraic standing waves of equation \eqref{triple power equation} are of the form
\[
e^{i\theta_0}\phi(\cdot-x_0),
\]
for some $\theta\in\R$ and $x_0\in\R^n$. 
\end{proposition}

\begin{remark}
\label{remark nice}
\begin{itemize}
\item[(1)] In case DFD, we only obtain the result on existence of algebraic standing waves when $n=1$ (see Proposition \ref{pro1}). The variational characterization of algebraic standing waves and stability or instability of these solutions are open problems, even in dimension one. 
\item[(2)] By using similar arguments as in \cite[Proof of Proposition 3.5]{FuHa21}, we prove that the algebraic standing waves in higher dimensions ($n=2,3$) are also in $H^1(\R^n)$.
\item[(3)] By scaling invariance of \eqref{triple power equation}, we may assume $|a_1|=|a_3|=1$ without loss of generality. This assumption will be made throughout the rest of this paper. 
\end{itemize}
\end{remark}

Before stating the main result, we define the orbital stability and orbital instability of standing waves.
\begin{definition}
Let $u_{\omega}(t,x)=e^{i\omega t}\phi_{\omega}(x)$ be a standing wave solution of \eqref{triple power equation}. We say that this solution is orbitally stable if for all $\varepsilon>0$ there exists $\delta>0$ such that for each $u_0 \in H^1(\R^n)$ such that $\norm{u_0-\varphi_{\omega}}_{H^1}<\delta$ then the associated solution $u$ of \eqref{triple power equation} is global and satisfies 
\[
\inf_{\theta\in\R,\ y\in\R^n}\norm{u(t)-e^{i\theta}\varphi_{\omega}(\cdot-y)}_{H^1} <\varepsilon.
\]
Otherwise, $u_{\omega}$ is orbitally unstable. 
\end{definition}

Our main result is the following.
\begin{theorem}
\label{Theorem 1}
Let $n=1,2,3$. Assume that the parameters of \eqref{triple power equation} satisfy DDF or DFF when $n=2,3$ or DFF and $a_2<\frac{32}{15\sqrt{6}}$ when $n=1$. Then the algebraic standing wave $\phi$ given as in Proposition \ref{pro1} and Proposition \ref{pro variational characteristic} is orbitally unstable in $H^1(\R^n)$.
\end{theorem}

\begin{remark}
In \cite{FuHa21}, the authors proved instability of algebraic standing waves in the case $a_1=-1$, $a_2=0$, $a_3=1$. Our result can be seen as a small extension of this result for $a_2\neq 0$. We predict that a similar result holds in the case of multiple nonlinearity of the form $-|u|u+a_1|u|^{p_1}u+...+a_m|u|^{p_m}u+|u|^3u$, for $1<p_1<...<p_m<3$ are given and some conditions on $a_1,...,a_m$. More general, we may expect a similar result with more general nonlinearity of form $-|u|^{p_1}+a_2|u|^{p_2}u+...+a_m|u|^{p_m}+|u|^{p_{m+1}}$, where $0<p_1<p_2<...<p_{m+1}<\frac{4}{n-2}$ . However, we do not consider these cases in this paper.
\end{remark}

When the highest power is $L^2(\R^n)$-supercritical, standing waves usually are unstable by blow up (see \cite{Co09} for focusing simple power nonlinearity, \cite{FuHa21} for double power nonliearity). In our case, this is conserved. We prove the following result.
\begin{theorem}
\label{theorem 2}
Let $n=2,3$, D*F and $a_2>-\varepsilon$ when $n=3$, for $\varepsilon>0$ small enough. Let $\omega>0$ and $\phi_{\omega}$ be a ground state of \eqref{triple power equation} i.e $\phi_{\omega}$ is a minimizer of the following variational problem
\[
\inf\{S_{\omega}(v): v\in H^1(\R^n)\setminus \{0\},K_{\omega}(v)=0\},
\] 
where
\begin{align}
S_{\omega}(v)&:=\frac{1}{2}\norm{\nabla v}^2_{L^2}+\frac{\omega}{2}\norm{v}^2_{L^2}+\frac{1}{3}\norm{v}^3_{L^3}-\frac{a_2}{4}\norm{v}^4_{L^4}-\frac{1}{5}\norm{v}^5_{L^5}, \label{S omega}\\
K_{\omega}(v)&:=\left\langle S_{\omega}'(v),v\right\rangle= \norm{\nabla v}^2_{L^2}+\omega\norm{v}^2_{L^2}+\norm{v}^3_{L^3}a_2\norm{v}^4_{L^4}-\norm{v}^5_{L^5}.\label{K omega}
\end{align}
Then the standing waves $e^{i\omega t}\phi_{\omega}(x)$ of \eqref{triple power equation} are unstable by blow up. 
\end{theorem} 

The rest of this paper is organized as follows. In Section \ref{sec1}, we find the region of parameters $a_1,a_2,a_3$ in which there exist solutions of the elliptic equation \eqref{SW algebraic profile}. Specially, in one dimension, all solution of \eqref{SW algebraic profile} are algebraic standing waves. In Section \ref{sec2}, we establish the variational characterization of solutions given in Section \ref{sec1}. The existence of algebraic standing waves in high dimensions is also proved in section \ref{sec2}. In Section \ref{sec3}, we prove instability of algebraic standing waves Theorem \ref{Theorem 1} and instability by blow up of standing waves in the case of positive frequency Theorem \ref{theorem 2}.     

\section{Existence of solution of the elliptic equation}
\label{sec1}
First, we find the region of parameters $a_1$, $a_2$, $a_3$ in which there exist solutions of \eqref{SW algebraic profile}.  
\subsection{In dimension one}

Let $n=1$. To study the existence of algebraic standing waves, we use the following lemma (see \cite{BeLi83}, \cite[Proposition 2.1]{LiTsIa21})

\begin{lemma}\label{old version of existence}
Let $g$ be a locally Lipschitz continuous function with $g(0)=0$ and let $G(t)=\int_0^tg(s)\,ds$. A necessary and sufficient condition for the existence of a solution $\phi$ of the problem 
\begin{equation}
\label{equation 1}
\begin{cases}
\phi\in C^2(\R), \quad \lim_{x\rightarrow\pm\infty}\phi(x)=0, \quad \phi(0)>0,\\
\phi_{xx}+g(\phi)=0,
\end{cases}
\end{equation}
is that $c=\inf\left\{t>0:G(t)=0\right\}$ exists, $c>0$, $g(c)>0$. 
\end{lemma}

Using Lemma \ref{old version of existence}, we have the following result.
\begin{lemma}\label{uniqueness of algebraic SW}
Let $g(u)=a_1u^2+a_2u^3+a_3u^4$ be such that $g$ satisfies the assumptions of Lemma \ref{old version of existence} for some $a_1,a_2,a_3\in\R$. Then there exists a positive solution $\phi$ of \eqref{equation 1}. Moreover, all complex valued solutions of \eqref{equation 1} are of form:
\[
e^{i\theta_0}\phi(x-x_0),
\] 
for some $\theta_0,x_0\in\R$.
\end{lemma}
\begin{proof}
By Lemma \ref{old version of existence}, there exists a real valued solution $\phi$ of \eqref{equation 1}. We have
\begin{equation}\label{eq2}
\phi_{xx}+a_1\phi^{2}+a_2\phi^{3}+a_3\phi^{4}=0.
\end{equation}
Since $\lim_{x\rightarrow\pm\infty}\phi(x)=0$, there exists $x_0$ such that $\phi_x(x_0)=0$. Multiplying two sides of \eqref{eq2} by $\phi_x$ and noting that $\lim_{x\rightarrow\infty}\phi(x)=0$ we obtain
\begin{equation}\label{eq3}
\frac{1}{2}\phi_x^2+\frac{a_1}{3}\phi^{3}+\frac{a_2}{4}\phi^{4}+\frac{a_3}{5}\phi^{5}=0.
\end{equation} 
We see that $\phi$ is not vanishing on $\R$. Indeed, if $\phi(x_1)=0$ for some $x_1\in\R$ then $\phi_x(x_1)=0$ by \eqref{eq3}. Thus, $\phi \equiv 0$ by uniqueness of solutions of \eqref{eq3} which is a contradiction. Then, we can assume that $\phi>0$. 

The value $\phi(x_0)$ is a positive solution of $G(u)=\frac{a_1}{3}u^{3}+\frac{a_2}{4}u^{4}+\frac{a_3}{5}u^{5}=0$. Since $g$ satisfies the condition in Lemma \ref{old version of existence}, it follows that $G(u)=0$ has a first positive solution $c$ such that $g(c)>0$. If $\phi(x_0) \neq c$ then $G$ has another positive zero $d>c$ such that $d=\phi(x_0)$. By continuity of $\phi$, there exists $x_1>x_0$ such that $\phi(x_1)=c$ and by \eqref{eq3} $\phi_x(x_1)=0$. This conclusion implies that every positive solution of \eqref{eq2} has a critical point such that the value of solution at this point equals to $c$.

Let $u$ be a complex valued solution of \eqref{equation 1}. We prove that $u=e^{i\theta_0}\phi(x-x_0)$, for some $\theta_0,x_0\in\R$. We use similar arguments as in \cite[Theorem 8.1.4]{Ca03}. Multiplying the equation by $\overline{u}_x$ and taking real part, we obtain:
\[
\frac{d}{dx}\left(\frac{1}{2}|u_x|^2+\frac{a_1}{3}|u|^{3}+\frac{a_2}{4}|u|^{4}+\frac{a_3}{5}|u|^{5}\right) =0.
\] 
Thus,
\[
\frac{1}{2}|u_x|^2+\frac{a_1}{3}|u|^{3}+\frac{a_2}{4}|u|^{4}+\frac{a_3}{5}|u|^{5}=K.
\]
Using $\lim_{x\rightarrow\pm\infty}u(x)=0$ we have $K=0$. In particular, $|u|>0$. Indeed, if $u$ vanishes then $u_x$ vanish at the same point, hence, $u \equiv 0$. Therefore, we may write $u=\rho e^{i\theta}$, where $\rho>0$ and $\rho,\theta\in C^2(\R)$. Substituting $u=\rho e^{i\theta}$ in \eqref{equation 1} we have $2\rho_x\theta_x+\rho\theta_{xx}=0$ which implies there exists $\tilde{K} \in\R$ such that $\rho^2\theta_x=\tilde{K}$ and so $\theta_x=\frac{\tilde{K}}{\rho^2}$. Moreover, since $|u_x|$ is bounded, it follows that $\rho^2\theta_x^2$ is bounded. Thus, $\frac{\tilde{K}^2}{\rho^2}$ is bounded. Since $\rho(x) \rightarrow 0$ as $x \rightarrow \infty$, we have $\tilde{K} =0$. Thus, since $\rho>0$ we have $\theta \equiv \theta_0$ for some $\theta_0\in\R$. Thus $u=e^{i\theta_0}\rho$. Since $\rho$ is a positive solution of \eqref{eq2}, there exists $x_2\in\R$ such that $\rho(x_2)=c$ and $\rho_x(x_2)=0$. Thus, by uniqueness of solution of \eqref{eq2}, there exists $x_3 \in\R$ such that $\rho(x)=\phi(x-x_3)$ and $u=e^{i\theta_0}\phi(x-x_3)$. This implies the desired result. 
\end{proof}

Moreover, we have the following result.
\begin{lemma}\label{SW algebraic in H1}
Let $g$ and $\phi$ be as in Lemma \ref{uniqueness of algebraic SW}. Then $\phi \in H^1(\R)$.
\end{lemma}

\begin{proof}
Firstly, since $g$ satisfies the assumption of Lemma \ref{old version of existence}, we have $a_1<0$ (see the arguments in the proof of Proposition \ref{pro1}). As in the proof of Lemma \ref{uniqueness of algebraic SW}, up to a translation, we may assume that $\phi_x(0)=0$ and let $c=\phi(0)$. Then $\phi$ is an even function of $x$. Furthermore, $\phi$ satisfies 
\begin{equation}\label{eq4}
\frac{1}{2}\phi_x^2+G(\phi)=0.
\end{equation}
Moreover, $\phi_{xx}(0)=-g(\phi(0))=-g(c)<0$. Therefore, there exists $a>0$ such that $\phi_x<0$ on $(0,a)$. We claim that $a=\infty$. Otherwise, there would exists $b>0$ such that $\phi_x<0$ on $(0,b)$ and $\phi_x(b)=0$. Thus, $\phi(b)<c$ is a positive zero of $G$. This is a contradiction since $c$ is the first positive solution of $G$. Hence, $\phi_x<0$ on $(0,\infty)$. Thus, there exists $0\leq l<c$ such that $\lim_{x \rightarrow\infty}\phi(x)=l$. In particular, there exists $x_m \rightarrow \infty$ such that $\phi_x(x_m) \rightarrow 0$ as $m\rightarrow\infty$. Passing to the limit in \eqref{eq4} we have $G(l)=0$ and hence $l=0$ by definition of $c$. Therefore $\phi$ decreases to $0$, as $x\rightarrow\infty$. Thus, from \eqref{eq4}, for $|x|$ large enough, we have
\[
\phi_x^2 \approx -\frac{a_1}{3} \phi^{3}.
\]    
Then 
\[
-\phi_x \approx c \phi^{\frac{3}{2}}, \text{ for some } c>0.
\]
Thus, for $|x|$ large enough, we have
\[
0\geq \phi_x+c\phi^{\frac{3}{2}}.
\]
It follows that $\phi\leq \frac{1}{(cx+d)^{2}}$ for some $c,d>0$. Hence $\phi \in L^1(\R) \cap L^{\infty}(\R)$, especially $\phi\in L^2(\R)$. Combining this and \eqref{eq4}, we obtain that $\phi_x \in L^2(\R)$. Thus, $\phi \in H^1(\R)$, this completes the proof of Lemma \ref{SW algebraic in H1}.  
\end{proof}

Now, we comeback to the proof of Proposition \ref{pro1}.
\begin{proof}[Proof of Proposition \ref{pro1}]
A solution of \eqref{SW algebraic profile} in the space $X$ satisfies
\begin{equation}\label{eq algebraic standing wave}
u_{xx}+g(u)=0, \quad u\in C^2(\R), \quad \text{ and } \lim_{x\rightarrow\pm\infty}u(x)=0,
\end{equation}
From Lemma \ref{old version of existence}, the necessary condition for existence of solutions of \eqref{eq algebraic standing wave} is $a_1<0$. Indeed, let $c$ is the first positive root of $G(u)$ then $G'(c)=g(c)>0$. Thus, $G$ do not change sign on $(0,c)$ and is increasing in a neighborhood of $c$. It follows that $G<0$ on $(0,c)$ and hence $a_1<0$. \\
To conclude the existence of solution of \eqref{eq algebraic standing wave}, we consider the three cases DDF, DFF, DFD. In the case DDD we have $G<0$ on $(0,\infty)$, therefore there is no solution of \eqref{eq algebraic standing wave}. \\
In the case DDF (i.e $a_1=-1$, $a_2<0$, $a_3=1$), we have
\begin{align*}
g(s)&= -s^{2}+a_2 s^{3} +s^{4},\\
G(s)&=-\frac{1}{3}s^{3}+\frac{a_2}{4}s^{4}+\frac{1}{5}s^{5}.
\end{align*}    
Thus ,
\[
c=\frac{-\frac{a_2}{4}+\sqrt{\frac{a_2^2}{16}+\frac{4}{15}}}{\frac{2}{5}},
\]       
and $g(c)=c^2(c^2+a_2 c-1)$. It easy to check that $c$ is larger than the largest root of $x^2+a_2 x-1$. Thus, $g(c)>0$. It follows that in case DDF, there exists a solution of \eqref{eq algebraic standing wave}.\\
By similar arguments, in the case DFF, \eqref{eq algebraic standing wave} has a solution. In the case DFD, \eqref{eq algebraic standing wave} has a solution if and only if $a_2>\frac{8}{\sqrt{15}}$. \\
Let $\phi$ be a solution of \eqref{eq algebraic standing wave}. From Lemma \ref{uniqueness of algebraic SW} all solution of \eqref{eq algebraic standing wave} are of the form $e^{i\theta}\phi(x-x_0)$, and belong to $H^1(\R)$ by Lemma \ref{SW algebraic in H1}. Thus, they are all algebraic standing waves of \eqref{triple power equation}. This completes the proof of Proposition \ref{pro1}. 
\end{proof}

\subsection{In higher dimensions}

In this section, we prove existence and uniqueness of a radial positive solution of \eqref{SW algebraic profile} when $a_1=-1$, $a_3=1$ and $n=2,3$. The existence result is a consequence of the following theorem.
\begin{theorem}[\cite{BeLiPe81},Theorem \RNum{1}.1]\label{exstence of SW in high dimensions}
Let $g$ be a locally Lipschitz continuous function from $\R^+$ to $\R$ with $g(0)=0$, satisfying
\begin{itemize}
\item[(1)] $\alpha=\inf\{\zeta>0,g(\zeta) \geq 0\}$ exists, and $\alpha>0$.
\item[(2)] There exists a number $\zeta>0$ such that $G(\zeta)>0$, where
\[
G(t)=\int_0^tg(s)\,ds.
\]
Define $\zeta_0=\inf\{\zeta>0,G(\zeta)>0\}$. Then, $\zeta_0$ exists, and $\zeta_0>\alpha$.
\item[(3)] $\lim_{s\downarrow\alpha}\frac{g(s)}{s-\alpha}>0$.
\item[(4)] $g(s)>0$ for $s \in (\alpha,\zeta_0]$.
Let $\beta=\inf\{\zeta>\zeta_0,g(\zeta)=0\}$. Then, $\zeta_0<\beta\leq \infty$.
\item[(5)] If $\beta=\infty$ then $\frac{g(s)}{s^l}=0$, with $l<\frac{n+2}{n-2}$, (If $n=2$, we may choose for $l$ just any finite real number).
\end{itemize} 
Then there exists a number $\zeta\in (\zeta_0,\beta)$ such that the solution $u\in C^2(\R^+)$ of the Initial Value problem
\begin{equation*}
\begin{cases}
-u''-\frac{n-1}{r}u'=g(u), \text{ for } r>0,\\
u(0)=\zeta, \ u'(0)=0
\end{cases}
\end{equation*}
has the properties: $u>0$ on $\R^+$, $u'<0$ on $\R^+$ and 
\[
\lim_{r\rightarrow\infty}u(r)=0.
\] 
\end{theorem}

In our case, we have
\begin{align}
g(s)&= -s^2+a_2s^3+s^4, \label{our case function g}\\
G(s)&=\frac{-1}{3}s^3+\frac{a_2}{4}s^4+\frac{1}{5}s^5.\label{our case function G}
\end{align}
It is easy to check that the function $g$ and $G$ satisfy the conditions of Theorem \ref{exstence of SW in high dimensions} when $n=2,3$ with $\alpha=\frac{-a_2+\sqrt{a_2^2+4}}{2}$ (the positive zero of $g$), $\zeta_0=\frac{-a_2+\sqrt{a_2^2+\frac{64}{15}}}{\frac{8}{5}}$ (the positive zero of $G$), $\beta=\infty$ and $4<l<5$ when $n=3$ and $l>4$ when $n=2$. Thus, in high dimensions ($n=2,3$), there exists a decreasing radial positive solution of \eqref{SW algebraic profile}.   

The uniqueness of a radial positive solution is obtained by following result.


\begin{theorem}[\cite{PuSe98},Theorem 1]\label{thm uniqueness 2}
Let us consider, for $n \geq 2$, the following equation
\begin{equation}\label{eq1000}
\Delta u+g(u)=0, 
\end{equation}
where $g$ satisfies the following conditions:
\begin{itemize}
\item[(a)] $g$ is continuous on $[0,\infty)$ and $g(0)=0$,
\item[(b)] $g$ is a $C^1$-function on $(0,\infty)$,
\item[(c)] There exists $a>0$ such that $g(a)=0$ and
\begin{align*}
g(u)<0 &\text{ for } 0<u<a,\\
g(u)>0 &\text{ for } u>a.
\end{align*}
\item[(d)] $\frac{d}{du}\left[\frac{G(u)}{g(u)}\right] \geq \frac{n-2}{2n}, \text{ for } u>0, u\neq a
$, where $G(s)=\int_0^s f(\tau)d\tau$.
\end{itemize}
Then \eqref{eq1000} admits at most one radial positive solution.
\end{theorem}

The function $g$ given in \eqref{our case function g} satisfies conditions (a), (b), (c) of Theorem \ref{thm uniqueness 2} for $a$ the positive root of $g$. When $n=2,3$, the condition (d) is satisfied if only if 
\begin{equation}\label{eq star}
\frac{d}{ds}\left[\frac{\frac{1}{5}s^3+\frac{a_2}{4}s^2-\frac{1}{3}s}{s^2+a_2s-1}\right] \geq \frac{n-2}{2n}, \text{ for } s>0,s \neq a.
\end{equation} 
We prove that \eqref{eq star} holds. We only need to show that 
\[
\frac{d}{ds}\left[\frac{\frac{1}{5}s^3+\frac{a_2}{4}s^2-\frac{1}{3}s}{s^2+a_2s-1}\right] \geq \frac{1}{6}, \text{ for } s \neq a.
\]
This is equivalent to
\[
\frac{1}{5}s^4+\frac{2a_2}{5}s^3+\left(\frac{a_2^2}{2}+\frac{2}{5}\right)-a_2s+1 \geq 0,
\]
which is true for all $s >0$, $a_2 \in\R$ by the fact that
\[
\frac{1}{5}s^4+\frac{2a_2}{5}s^3+\left(\frac{a_2^2}{2}+\frac{2}{5}\right)-a_2s+1=\frac{1}{5}(s^2+a_2s)^2+\frac{3}{10}\left(a_2-\frac{5}{3}\right)^2+\frac{2}{5}s^2+\frac{1}{6}>0.
\]
Thus, there exists a unique radial positive solution of \eqref{SW algebraic profile} by Theorem \ref{thm uniqueness 2}. This completes the proof of Proposition \ref{pro2}.

\section{Variational characterization}

\label{sec2}
Let $n=1,2,3$. In this section, we prove Proposition \ref{pro variational characteristic}. By the assumption of Proposition \ref{pro variational characteristic}, we may pick $a_1=-1$ and $a_3=1$. We recall that $S,K,P$ are defined in \eqref{define of S omega}, \eqref{define of K omega} and \eqref{define of P}.

Let $\mathcal{M}$ and $\mathcal{K}$ be defined as \eqref{define of M} and \eqref{define of K omega}. First, as in \cite{FuHa21}, we prove that $\mathcal{M} $ is not empty. We set
\begin{align*}
J(v)&=\frac{1}{4}\norm{\nabla v}^2_{L^2}+\frac{1}{12}\norm{v}^3_{L^3}+\frac{1}{20}\norm{v}^5_{L^5},\\
\end{align*}
which is well defined on $X$. The functional $S$ is rewritten as
\begin{align*}
S(v)&=\frac{1}{2}K(v)-\frac{1}{6}\norm{v}^3_{L^3}+\frac{a_2}{4}\norm{v}^4_{L^4}+\frac{3}{10}\norm{v}^5_{L^5},\\
S(v)&= \frac{1}{4}K(v)+J(v).
\end{align*}
We can rewrite $\mu$ as
\begin{equation}
\label{eq new of mu}
\mu =\inf\{J(v): v\in \mathcal{K}\}.
\end{equation}
\begin{lemma}\label{lm property of mu}
Let $v \in H^1(\R^n)$. If $K(v)<0$ then $\mu<J(v)$. In particular,
\begin{equation}
\label{eq new of mu 2}
\mu=\inf\{J(v): v\in X\setminus\{0\}, K(v)\leq 0\}.
\end{equation}
\end{lemma}

\begin{proof}
Since $K(v)<0$ and $K(\lambda v) > 0$ if $\lambda>0$ small enough, there exists $\lambda_1\in (0,1)$ such that $K(\lambda_1 v)=0$. Therefore, by \eqref{eq new of mu} and since the function $\lambda \mapsto J(\lambda v)$ on $(0,\infty)$ is increasing, we have
\[
\mu \leq J(\lambda_1 v) < J(v).
\]
This completes the proof. 
\end{proof}

\begin{lemma}\label{mu positive}
The following is true:
\[
\mu>0.
\]
\end{lemma}

\begin{proof}
Let $v \in \mathcal{K}$. By using the Gagliardo-Nirenberg inequalities, for some $\theta \in (0,5)$ and $\tilde{\theta}\in (0,4)$, we have 
\begin{align*}
\norm{v}^{5}_{L^5} & \lesssim \norm{\nabla v}_{L^2}^{\theta}\norm{v}_{L^3}^{5-\theta} \leq C_1\norm{\nabla v}_{L^2}^{5}+C_2\norm{v}^5_{L^3},\\
\norm{v}^4_{L^4} &\lesssim \norm{\nabla v}_{L^2}^{\tilde{\theta}}\norm{v}_{L^3}^{4-\tilde{\theta}} \leq C_3 \norm{\nabla v}_{L^2}^4+C_4\norm{v}^4_{L^3},
\end{align*}
we have
\[
0=K(v)\geq (1-C_1\norm{\nabla v}^{3}_{L^2}-|a_2|C_3\norm{\nabla v}^2_{L^2})\norm{\nabla v}^2_{L^2}+(1-C_2\norm{v}_{L^3}^2-|a_2|C_4\norm{v}_{L^3})\norm{v}^3_{L^3},
\]
It follows that $1 \leq C_1\norm{\nabla v}^3_{L^2}+|a_2|C_3\norm{\nabla v}^2_{L^2}\leq C\norm{\nabla v}^3_{L^2}+\frac{1}{2}$ or $1 \leq C_2\norm{v}^2_{L^3}+|a_2|C_4\norm{v}_{L^3}\leq \tilde{C}\norm{v}^2_{L^3}+\frac{1}{2}$, for some $C,\tilde{C}>0$. Hence, $\norm{\nabla v}_{L^2}$ or $\norm{v}^3_{L^3}$ bounded below by some constant. In two cases, $J(v)$ is bounded below by some constant. Combining with \eqref{eq new of mu} we have the conclusion.
\end{proof}

We need the following results.
\begin{lemma}[\cite{BeFrVi14,Li83}]\label{lm converges weakly}
Let $p \geq 1$. Let $(f_n)$ be a bounded sequence in $\dot{H}^{1}(\R^n)\cap L^{p+1}(\R^n)$. Assume that there exists $q \in (p,2^{*}-1)$ such that $\limsup_{n\rightarrow\infty}\norm{f_n}_{L^{q+1}}>0$. Then there exist $(y_n)\subset \R^n$ and $f\in \dot{H}^{1}(\R^n)\cap L^{p+1}(\R^n)\setminus \{0\}$ such that $(f_n(\cdot-y_n))$ has a subsequence that converges to $f$ weakly in $\dot{H}^{1}(\R^n)\cap L^{p+1}(\R^n)$. 
\end{lemma}

\begin{lemma}[\cite{BrLi83}]
\label{lm Lieb}
Let $1 \leq r< \infty$. Let $(f_n)$ be a bounded sequence in $L^r(\R^n)$ and $f_n \rightarrow f$ a.e in $\R^n$ as $n\rightarrow\infty$. Then
\[
\norm{f_n}^r_{L^r}-\norm{f_n-f}^r_{L^r}-\norm{f}^r_{L^r} \rightarrow 0,
\] 
as $n\rightarrow\infty$.
\end{lemma}

Now, we comeback to prove the set $\mathcal{M}$ is not empty. 
\begin{lemma}\label{M not empty}
If $(v_n)\in X$ is a minimizing sequence for $\mu$, that is,
\[
K(v_n) \rightarrow 0, \quad S(v_n) \rightarrow \mu,
\]
then there exist $(y_n) \subset \R^n$, a subsequence $(v_{n_j})$, and $v_0 \in X\setminus \{0\}$ such that $v_{n_j}(\cdot-y_{n_j}) \rightarrow v_0$ in $X$. In particular, $v_0 \in \mathcal{M}$.
\end{lemma}

\begin{proof}
Since $K(v_n) \rightarrow 0$ and $S(v_n) \rightarrow \mu$, we have 
\begin{align}
J(v_n) &\rightarrow \mu,  \label{converge to mu}\\
\frac{-1}{6}\norm{v}^3_{L^3}+\frac{a_2}{4}\norm{v}^4_{L^4}+\frac{3}{10}\norm{v}^5_{L^5} &\rightarrow \mu. \label{convergen to mu 2}
\end{align}
From \eqref{converge to mu}, we infer that $(v_n) $ is bounded in $X$. Also, since $\mu>0$ by Lemma \ref{mu positive} and the Gagliardo-Nirenberg inequality $\norm{v}^5_{L^5} \lesssim \norm{\nabla v}^5_{L^2} +\norm{v}^5_{L^4}$, we have $\limsup_{n\rightarrow\infty}\norm{v_n}_{L^4} >0$. Then, by Lemma \ref{lm converges weakly} there exist $(y_n) \subset \R^n$ and $v_0\in X\setminus\{0\}$ and a subsequence of $(v_n(\cdot-y_n))$, which we still denote by the same notation, such that $v_n(\cdot-y_n) \rightharpoonup v_0$ weakly in $X$. we put $w_n:=v_n(\cdot-y_n)$.
   
We can assume that $w_n \rightarrow v_0$ a.e in $\R^n$ and we prove that $w_n \rightarrow v_0$ strongly in $X$. By Lemma \ref{lm Lieb}, we have
\begin{align}
J(w_n)-J(w_n-v_0)\rightarrow J(v_0), \label{converge of J}\\
K(w_n)-K(w_n-v_0)\rightarrow K(v_0). \label{converge of K}
\end{align}
Since $J(v_0)>0$ by $v_0 \neq 0$, it follows from \eqref{converge of J} and \eqref{converge to mu} that 
\[
\lim_{n\rightarrow\infty}J(w_n-v_0)=\lim_{n\rightarrow\infty}J(w_n)-J(v_0)<\lim_{n\rightarrow\infty}J(w_n)=\mu.
\]
From this and \eqref{eq new of mu 2} we have $K(w_n-v_0)>0$ for $n$ large. Thus, since $K(v_n) \rightarrow 0$ and \eqref{converge of K} we obtain $K(v_n)\leq 0$. By \eqref{eq new of mu 2} and weak lower semicontinuity of the norms, we have
\[
\mu \leq J(v_0) \leq \lim_{n\rightarrow\infty}J(w_n)=\mu.
\]
Combining with \eqref{converge of J} imply that $J(w_n-v_0) \rightarrow 0$ thus, $w_n \rightarrow v_0$ strongly in $X$. This completes the proof.
\end{proof}

\begin{proof}[Proof of Proposition \ref{pro variational characteristic}]
Firstly, we prove the variational characterization of $\phi$ as follows
 \[
 S(\phi)=\mu.
 \]
This means that $\phi$ is a minimizer of \eqref{define of mu}. From Lemma \ref{M not empty}, we have $\mathcal{M} \neq \emptyset$. Let $\varphi\in \mathcal{M}$. We divide the proof of this to three steps.\\
Step 1. There exists $\theta \in \R$ such that $e^{i\theta}\varphi$ is a positive function.\\
We use similar arguments as in \cite[Lemma 2.10]{FuHa21}. Put $v:=|\Re\varphi|$, $w:=|\Im\varphi|$ and $\psi:=v+iw$. By a phase modulation, we may assume that $v\neq 0$. 

Since $|\psi|=|\varphi|$ and $|\nabla \psi|=|\nabla\varphi|$, we have $K(\psi)=K(\varphi)$ and $S(\psi)=S(\varphi)$. Thus, $\psi \in \mathcal{M}$. Then, there exists $\gamma \in\R$ such that 
\[
S'(\psi)=\gamma K'(\psi).
\]  
Hence, 
\begin{equation}\label{eq not important}
\gamma \left<K'(\psi),\psi\right>=\left<S'(\psi),\psi\right>=K(\psi)=0.
\end{equation}
Moreover, using $K(\psi)=0$ we have
\begin{align*}
\left<K'(\psi),\psi\right>&=\partial_{\lambda}K(\lambda\psi)\vert_{\lambda=1}\\
&=\partial_{\lambda}K(\lambda\psi)\vert_{\lambda=1}-4K(\psi)\\
&=(2\norm{\nabla \psi}^2_{L^2}+3\norm{\psi}^3_{L^3}-4a_2\norm{\psi}^4_{L^4}-5\norm{\psi}^5_{L^5})-4(\norm{\nabla \psi}^2_{L^2}+\norm{\psi}^3_{L^3}-a_2\norm{\psi}^4_{L^4}-\norm{\psi}^5_{L^5})\\
&=-2\norm{\nabla \psi}^2_{L^2}-\norm{\psi}^3_{L^3}-\norm{\psi}^5_{L^5}<0.
\end{align*}
Combining with \eqref{eq not important}, we deduce $\gamma=0$. Thus, $S'(\psi)=0$. Hence, $v$ solves the following equation
\[
(-\Delta + |\varphi|-a_2|\varphi|^2-|\varphi|^3)v=0.
\]
Since $v$ is nonnegative and not identically equal to zero, using \cite[Theorem 9.10]{LiLo01}, we infer that $v$ is positive function. Furthermore, since $K(|\psi|) \leq K(\psi)$ and $S(|\psi|) \leq S(\psi)$, it follows from Lemma \ref{lm property of mu} we have $K(|\psi|)=K(\psi)$ and $S(|\psi|)=S(\psi)$. Then, $\norm{\nabla|\psi|}_{L^2}=\norm{\nabla \psi}_{L^2}$. By \cite[Theorem 7.8]{LiLo01}, there exists a constant $c$ such that $w=cv$ for some $c \geq 0$. 

Since $v$ is continuous and positive, $\Re\varphi$ and $\Im\varphi$ do not change sign. Then, there exist constants $\lambda=\pm 1$ and $\eta \in\R$ such that $\Re\varphi=\lambda v$ and $\Im\varphi=\eta v$. Taking $\theta\in\R$ such that $e^{-i\theta}=\frac{\lambda+i\eta}{|\lambda+i\eta|}$, we have $e^{i\theta}\varphi=e^{i\theta}(\lambda+i\eta)v=|\lambda+i\eta|v$. This completes the step 1.   \\
Step 2. Radial symmetry of minimizer.\\
Since \cite[Theorem 1]{LiNi93}, there exists $y \in \R^n$ such that $e^{i\theta}\varphi(\cdot-y)$ is a radial and decreasing function. \\
Step 3. Conclusion.\\
Since $\phi$ and $e^{i\theta}\varphi(\cdot-y)$ are positive radial solutions of \eqref{SW algebraic profile}, using Proposition \ref{pro2}, we obtain 
\[
\phi=e^{i\theta}\varphi(\cdot-y),
\]
Thus, $S(\phi)=S(\varphi)=\mu$, $\phi\in\mathcal{M}$ and each element of $\mathcal{M}$ is of form $e^{i\theta}\phi(\cdot-x_0)$ for some $\theta,x_0\in\R$. 

It remains to classify all algebraic standing waves of \eqref{triple power equation}. We only need to prove that $\mathcal{G}=\mathcal{M} \neq \emptyset$, where $\mathcal{G}$ and $\mathcal{M}$ are defined in \eqref{define of G} and \eqref{define of M}, respectively. We use similar arguments as in \cite[Proof of Theorem 2.1]{FuHa21}. We divide the proof of this in two steps.\\
Step 1. $\mathcal{M} \subset \mathcal{G}$.\\
Let $\psi \in \mathcal{M}$. Then, $S'(\psi)=0$. Now, we show that $\psi\in\mathcal{G}$. Let $v\in X\setminus\{0\}$ such that $S'(v)=0$. From $K(v)=\left<S'(v),v\right>=0$ and by definition of $\mathcal{M}$, we have $S(\psi) \leq S(v)$. Thus, $\psi\in\mathcal{G}$ and $\mathcal{M}\subset\mathcal{G}$.\\
Step 2. $\mathcal{G}\subset\mathcal{M}$ and conclusion.\\
Let $\psi \in \mathcal{G}$. Then $K(\psi)=\left<S'(\psi),\psi\right>=0$. As the above, $\phi \in\mathcal{M}$. As in step 1, $\phi\in\mathcal{G}$. Therefore, $S(\psi)=S(\phi)=\mu$, which implies $\psi \in \mathcal{M}$. Thus $\mathcal{G}\subset\mathcal{M}$, which completes the proof of Proposition \ref{pro variational characteristic}.    
\end{proof}



It turns out that the algebraic standing waves of \eqref{triple power equation} in high dimensions ($n=2,3$) belongs to $H^1(\R^n)$. To prove this, we need the following lemma (see \cite[Lemma 3.4]{FuHa21}).
\begin{lemma}\label{lm 3.4}
Let $\varphi \in C^1([0,\infty))$ be a positive function. If there exist $\rho,A>0$such that
\[
\varphi'(r)+A\varphi(r)^{1+\rho} \leq 0, \text{ for all } r>0,
\]
then 
\[
\varphi(r) \leq \left(\frac{1}{\rho A r}\right)^{\frac{1}{\rho}}.
\]
\end{lemma}

\begin{proof}[Proof of Remark \ref{remark nice}(2)]
We use similar arguments as in \cite[Proof of Proposition 3.5]{FuHa21}. Firstly, we denote $\phi(r)$ as function of $\phi$ respect to variable $r=|x|$. Since $\phi$ is positive decreasing radial function, we have
\[
\norm{\phi}^3_{L^3} \geq \int_{|x|\leq R}|\phi|^3\,dx \geq |\mathcal{B}(R)||\phi(R)|^3=CR^n|\phi(R)|^{3},
\]
for all $R>0$. Hence, 
\[
\phi(x) \leq |x|^{-\frac{n}{3}}\norm{\phi}_{L^3}, \text{ for all } x\in\R.
\]
For $r>r_0$ large enough, we have
\[
|a_2|\phi^3+\phi^4 \leq \frac{1}{2}\phi^2,
\]
Since $\phi$ solves \eqref{SW algebraic profile} and is decreasing as a function of $r$, this implies
\[
\phi''(r) \geq \phi''(r)+\frac{n-1}{r}\phi'(r) = \phi^2-a_2\phi^3-\phi^4 \geq \frac{1}{2}\phi^2, \text{ for } r> r_0.
\]
Multiplying the two sides by $\phi'$ and integrating it on $[r,\infty)$, we get
\[
\phi'(r)^2 \geq \frac{1}{3}\phi^3, \text{ for } r\geq r_0.
\]
Since $\phi'<0$ we obtain that
\[
\phi'(r)+\sqrt{\frac{1}{3}}\phi^{\frac{3}{2}} \leq 0, \text{ for } r\geq r_0.
\]
By Lemma \ref{lm 3.4}, we deduce that
\[
\phi(r) \leq Cr^{-2}, \text{ for }r\geq r_0.
\]
Thus, $\phi \in L^2(\R^n)$, for $n=1,2,3$. From the proof of Proposition \ref{pro variational characteristic}, we have $\phi \in \mathcal{M}$. Hence, $|\nabla \phi| \in L^2(\R^n)$ and $\phi \in H^1(\R^n)$. This completes the proof.
\end{proof}

\begin{remark}
For each $\omega\geq 0$, let $\phi_{\omega}$ be a radial positive solution of \eqref{SW profile}. In the cases $\omega>0$, it is well known that $\phi_{\omega}$ exponential decays. In the case $\omega=0$, in special cases, we may find exactly solution of \eqref{SW algebraic profile} and hence shape decay of $\phi_0$. We may check that $(1+x^2)^{-1}$ solves 
\[
\phi''-6\phi^2+8\phi^3=0.
\]
Hence, in this case, $\phi_0 \approx x^{-2}$ when $x$ large. In our case, it is not easy to find exactly solution of \eqref{SW algebraic profile}. Given a lower bounded for $\phi_0$ when $x$ large is an unanswered question. 
\end{remark}

\section{Instability of algebraic standing waves}
\label{sec3}

Let $n=1,2,3$. In this section, we prove Theorem \ref{Theorem 1}. Throughout this section, we consider the case $DDF$ or $DFF$ and $a_2$ small. Then we may pick $a_1=-1$ and $a_3=1$. First, we prove the following result by using similar arguments as in \cite{Oh95} (see also \cite[Proof of Proposition 5.1]{FuHa21}).  
\begin{proposition}\label{pro instable 1}
Assume that
\begin{equation}
\label{assumption of instability one dimension}
\partial_{\lambda}^2 S(\phi^{\lambda})\vert_{\lambda=1}<0, \text{ where } v^{\lambda}(x):=\lambda^{\frac{n}{2}}v(\lambda x).
\end{equation}
Then the algebraic standing wave $\phi$ is unstable. 
\end{proposition}

We define a tube around the standing wave by
\[
\mathcal{N}_{\varepsilon}:=\left\{v\in H^1(\R^n): \inf_{(\theta,y)\in \R\times\R^n}\norm{v-e^{i\theta}\phi(\cdot-y)}_{H^1}<\varepsilon\right\}.
\]
\begin{lemma}\label{lemma 5.2}
Assume \eqref{assumption of instability one dimension} holds. Then there exist $\varepsilon_1,\delta_1\in (0,1)$ such that: For any $v\in \mathcal{N}_{\varepsilon_1}$ there exists $\Lambda(v) \in (1-\delta_1,1+\delta_1)$ such that
\[
\mu\leq S(v)+(\Lambda(v)-1)P(v).
\] 
\end{lemma}
\begin{proof}
First, we recall that $S$, $K$ and $P$ are defined as in \eqref{define of S omega}, \eqref{define of K omega} and \eqref{define of P}, respectively.

Since $\partial_{\lambda}^2S(\phi^{\lambda})\vert_{\lambda=1}<0$, by the continuity of the function
\[
(\lambda,v) \mapsto \partial_{\lambda}^2S(v^{\lambda}),
\]
there exist $\varepsilon_1,\delta_1\in (0,1)$ such that $\partial_{\lambda}^2S(v^{\lambda})<0$ for any $\lambda \in (1-\delta_1,1+\delta_1)$ and $v \in \mathcal{N}_{\varepsilon_1}$. Moreover, by the definition of $P$ we have 
\begin{equation}
\label{estimate}
S(v^{\lambda}) \leq S(v)+(\lambda-1)P(v),
\end{equation} 
for $\lambda\in (1-\delta_1,1+\delta_1)$ and $v \in \mathcal{N}_{\varepsilon_1}$.\\
Moreover, consider the map:
\[
(\lambda, v) \mapsto K(v^{\lambda})=\lambda^2\norm{\nabla v}^2_{L^2}+\lambda^{\frac{n}{2}}\norm{v}^3_{L^3}-a_2\lambda^{n}\norm{v}^4_{L^4}-\lambda^{\frac{3n}{2}}\norm{v}^5_{L^5}.
\]
Note that $K(\phi)=0$ and 
\begin{align*}
\partial_{\lambda}K(\phi^{\lambda})\vert_{\lambda=1}&=2\norm{\nabla\phi}^2_{L^2}+\frac{n}{2}\norm{\phi}^3_{L^3}-n a_2\norm{\phi}^4_{L^4}-\frac{3n}{2}\norm{\phi}^5_{L^5}.
\end{align*}
Thus, 
\begin{align*}
\partial_{\lambda}K(\phi^{\lambda})\vert_{\lambda=1}&=\partial_{\lambda}K(\phi^{\lambda})\vert_{\lambda=1}-5P(\phi)\\
&=-3\norm{\nabla\phi}^2_{L^2}-\frac{n}{3}\norm{\phi}^3_{L^3}+\frac{n a_2}{4}\norm{\phi}^4_{L^4}.
\end{align*}
Thus, in the case $a_2<0$, we have $\partial_{\lambda}K(\phi^{\lambda})\vert_{\lambda=1}<0$. In the case $a_2 \geq 0$, using $P(\phi)=0$, we have 
\begin{align*}
\frac{n a_2}{4}\norm{\phi}^4_{L^4}&= \norm{\nabla\phi}^2_{L^2}+\frac{n}{6}\norm{\phi}^3_{L^3}-\frac{3n}{10}\norm{\phi}^5_{L^5} \\
&\leq 3\norm{\nabla\phi}^2_{L^2}+\frac{n}{3}\norm{\phi}^3_{L^3},
\end{align*}  
hence we also have $\partial_{\lambda}K(\phi^{\lambda})\vert_{\lambda=1}<0$. In all cases, by the implicit function theorem, taking $\varepsilon_1$ and $\delta_1$ small enough, for any $v \in \mathcal{N}_{\varepsilon_1}$ there exists $\Lambda(v) \in (1-\delta_1,1+\delta_1)$ such that $\Lambda(\phi)=1$ and $K(v^{\Lambda(v)})=0$. Therefore, by definition of $\mu$ as in \eqref{define of mu} we obtain:
\[
\mu\leq S(v^{\Lambda(v)})\leq S(v)+(\Lambda(v)-1)P(v). 
\]
This completes the proof.
\end{proof}

Let $u_0\in \mathcal{N}_{\varepsilon}$ and $u(t)$ be the associated solution of \eqref{triple power equation}. We define the exit time from the tube $\mathcal{N}_{\varepsilon}$ by
\[
T^{\pm}_{\varepsilon}(u_0):= \inf\{t>0: u(\pm t)\notin \mathcal{N}_{\varepsilon}\}.
\]
We set $I_{\varepsilon}(u_0):=(-T_{\varepsilon}^{-}(u_0),T^{+}_{\varepsilon}(u_0))$.

\begin{lemma}\label{Lemma 5.3}
Assume \eqref{assumption of instability one dimension} holds and let $\varepsilon_1$ be given by Lemma \ref{lemma 5.2}. Then for any $u_0 \in \mathcal{B}\cap \mathcal{N}_{\varepsilon_1}$, where $\mathcal{B}$ is defined as in \eqref{define of B omega}, there exists $m=m(u_0)>0$ such that $P(u(t))\leq -m$ for all $t \in I_{\varepsilon_1}(u_0)$.
\end{lemma}

\begin{proof}
For $t \in I_{\varepsilon_1}(u_0)$, since $u(t) \in \mathcal{N}_{\varepsilon_1}$, it follows from Lemma \ref{lemma 5.2} that
\[
\mu-S(u_0) =\mu-S(u(t)) \leq -(1-\Lambda(u(t)))P(u(t)).
\] 
In particular, since $\mu>S(u_0)$ by $u_0\in \mathcal{B}$, we have $P(u(t)) \neq 0$. By continuity of the flow and $P(u_0)<0$ we obtain
\[
P(u(t))<0, \quad 1-\Lambda(u(t))>0.
\]
Therefore, we obtain
\[
-P(u(t)) \geq \frac{\mu-S(u_0)}{1-\Lambda(u(t))} \geq \frac{\mu-S(u_0)}{\delta_1}=:m(u_0)>0.
\]
This completes the proof.
\end{proof}

\begin{lemma}\label{lm go out time stable of solution}
Assume \eqref{assumption of instability one dimension} holds. Then $|I_{\varepsilon_1}|<\infty$ for all $u_0 \in \mathcal{B}\cap \mathcal{N}_{\varepsilon_1}\cap \Sigma$, where
\begin{equation}\label{defineSigma}
\Sigma = \left\{ v\in H^1(\R) : xv \in L^2(\R)\right\}.
\end{equation}
\end{lemma}

\begin{proof}
Let $u(t)$ be associated solution of $u_0 \in \mathcal{B}\cap \mathcal{N}_{\varepsilon_1}\cap \Sigma$. By the virial identity and Lemma \ref{Lemma 5.3} we have
\[
\frac{d^2}{dt^2}\norm{xu(t)}^2_{L^2}=8P(u(t))\leq -8m(u_0)
\]
for all $t \in I_{\varepsilon_1}(u_0)$, which implies $|I_{\varepsilon_1}(u_0)|<\infty$. This completes the proof.
\end{proof}

Let $\chi$ be a smooth cut-off function such that
\[
\chi(r):=\left\{\begin{matrix}
1 & \text{ if } 0\leq r\leq 1,\\
0 & \text{ if } r \geq 2.
\end{matrix}
\right. 
\]
and for $R>0$ define $\chi_R(x)=\chi\left(\frac{|x|}{R}\right)$.\\
The following is similar as in \cite[Lemma 4.5]{FuHa21}.
\begin{lemma}\label{lm chosing sequence}
There exists a function $R:(1,\infty) \rightarrow (0,\infty)$ such that $\chi_{R(\lambda)}\phi^{\lambda} \in \mathcal{B}\cap \Sigma \cap \mathcal{N}_{\varepsilon_1}$ for all $\lambda>1$ close to $1$, and that $\chi_{R(\lambda)}\phi^{\lambda} \rightarrow \phi$ in $H^1(\R^n)$ as $\lambda \downarrow 1$.
\end{lemma}

\begin{proof}
We divide the proof in three steps.\\
Step 1: Prove $\phi^{\lambda} \rightarrow \phi$ in $H^1(\R^n)$ as $\lambda \downarrow 1$.\\
We have 
\begin{align}
&\norm{\phi^{\lambda}-\phi}_{\dot{H}^1}+\norm{\phi^{\lambda}-\phi}_{L^2}\nonumber\\
&\leq  \norm{\lambda^{\frac{n}{2}}\phi(\lambda \cdot)-\phi(\lambda \cdot)}_{\dot{H}^1}+\norm{\phi(\lambda \cdot)-\phi(\cdot)}_{\dot{H}^1}+\norm{\lambda^{\frac{n}{2}}\phi(\lambda \cdot)-\phi(\lambda \cdot)}_{L^2}+\norm{\phi(\lambda \cdot)-\phi(\cdot)}_{L^2}\nonumber\\
&=(\lambda^{\frac{n}{2}}-1)(\lambda^{1-\frac{n}{2}}\norm{\phi}_{\dot{H^1}}+\lambda^{\frac{-n}{2}}\norm{\phi}_{L^2}) \label{first term}\\
&\quad+ \norm{\phi(\lambda \cdot)-\phi(\cdot)}_{\dot{H}^1}+\norm{\phi(\lambda \cdot)-\phi(\cdot)}_{L^2} \label{second term}.
\end{align}
The term \eqref{first term} converges to zero as $\lambda \rightarrow 1$. To prove the term \eqref{second term} converges to zero as $\lambda \rightarrow 1$, we prove for all $\phi \in L^p$, $1<p<\infty$, then the following holds
\[
\norm{\phi(\lambda x)-\phi(x)}_{L^p} \rightarrow 0, \text{ as } \lambda \rightarrow 1.
\]   
Indeed, we only need to consider $\phi$ is a integrable step function, by density of step function in $L^p(\R^n)$. It is sufficient to consider $\phi= \mathbbm{1}_{\mathcal{A}}$, for some measurable set $\mathcal{A}$. We have $\phi(\lambda x)=\mathbbm{1}_{\frac{1}{\lambda}\mathcal{A}}$ and
\begin{align*}
\norm{\phi(\lambda x)-\phi(x)}^p_{L^p}&=\norm{\mathbbm{1}_{\frac{1}{\lambda}\mathcal{A}}-\mathbbm{1}_{\mathcal{A}}}^p_{L^p}\\
&=\mu(\{\lambda x\in \mathcal{A}, x\not\in\mathcal{A}\}\cup \{x\in \mathcal{A},\lambda x\not\in\mathcal{A}\})\\
&\leq \mu(\mathcal{A})+\mu\left(\frac{1}{\lambda}\mathcal{A}\right)-2\mu\left(\mathcal{A}\cap\frac{1}{\lambda}\mathcal{A}\right),
\end{align*}    
this converges to zero when $\lambda$ converges to $1$. Thus, if we consider $\nabla \phi$ as a vector function then the term \eqref{second term} converges to zero as $\lambda$ converges to $1$. \\
Step 2: $\chi_{R(\lambda)}\phi^{\lambda} \rightarrow \phi$ as $\lambda\rightarrow 1$ for some function $R$.\\
Choosing $R: (1,\infty)\rightarrow (0,\infty)$ such that $R(\lambda) \rightarrow \infty$ as $\lambda \rightarrow 1$. Thus, for all $v \in H^1(\R^n)$, we have
\[
\chi_{R(\lambda)}v \rightarrow v, \text{ as } \lambda \rightarrow 1
\]
and $\chi_{R(\lambda)}\phi^{\lambda}\rightarrow \phi$ in $H^1(\R^n)$ as $\lambda \downarrow 1$, since step 1.  \\
Step 3: Conclusion.\\
We claim that $\phi^{\lambda} \in \mathcal{B}$ for $\lambda>1$ close to $1$. Since $\partial_{\lambda}S(\phi^{\lambda})\vert_{\lambda=1}=0$ and $\partial_{\lambda}^2S(\phi^{\lambda})\vert_{\lambda=1}<0$, there exists $\lambda_1>1$ such that $\partial_{\lambda}S(\phi^{\lambda})<0$ and $S(\phi^{\lambda})<\mu$ for $\lambda\in (1,\lambda_1)$. We see that $P(\phi^{\lambda})=\lambda \partial_{\lambda} S(\phi^{\lambda})<0$ for $\lambda \in (1,\lambda_1)$. Moreover, taking $\lambda_1$ close to $1$, we get $\phi^{\lambda} \in \mathcal{N}_{\varepsilon_1}$ for all $\lambda\in (1,\lambda_1)$. Since $\chi_{R(\lambda)}$ has compact support and $\norm{\chi_{R(\lambda)}\phi^{\lambda}-\phi^{\lambda}}_{H^1} \rightarrow 0$ as $\lambda\rightarrow 1$, we have $\chi_{R(\lambda)}\phi^{\lambda} \in \mathcal{B}\cap \mathcal{N}_{\varepsilon_1}\cap \Sigma$ for $\lambda$  close to $1$. This completes the proof.  
\end{proof}

\begin{proof}[Proof of Proposition \ref{pro instable 1}]
By Lemma \ref{lm chosing sequence}, there exists $R:(1,\infty) \rightarrow (0,\infty)$ such that $\chi_{R(\lambda)}\phi^{\lambda}\rightarrow \phi$ in $H^1(\R^n)$ as $\lambda \downarrow 1$. Moreover, $\chi_{R(\lambda)}\phi^{\lambda} \in \mathcal{B}\cap \Sigma \cap \mathcal{N}_{\varepsilon_1}$ for $\lambda>1$ close to $1$. Thus, by Lemma \ref{lm go out time stable of solution}, $|I_{\varepsilon_1}(\chi_{R(\lambda)}\phi^{\lambda})|<\infty$ for $\lambda>1$ close to $1$ and since $\chi_{R(\lambda)}\phi^{\lambda}  \rightarrow \phi$ as $\lambda \rightarrow 1$ in $H^1(\R^n)$ we have $ \phi$ is unstable. This completes the proof.
\end{proof}

\begin{proof}[Proof of Theorem \ref{Theorem 1}]
Using Proposition \ref{pro instable 1}, we only need to check the condition \eqref{assumption of instability one dimension}. We have 
\[
\partial_{\lambda}^2S(\phi^{\lambda})\vert_{\lambda=1}=\norm{\nabla\phi}^2_{L^2}+\frac{n(n-2)}{12}\norm{\phi}^3_{L^3}-\frac{n(n-1)a_2}{4}\norm{\phi}^4_{L^4}-\frac{3n(3n-2)}{20}\norm{\phi}^5_{L^5}.
\]
We divide into three cases. \\
Case $n=1$:\\
In this case, we have
\[
\partial_{\lambda}^2S(\phi^{\lambda})\vert_{\lambda=1}=\norm{\phi'}^2_{L^2}-\frac{1}{12}\norm{\phi}^3_{L^3}-\frac{3}{20}\norm{\phi}^5_{L^5}.
\]
In the case DDF, using $K(\phi)=0$ and $P(\phi)=0$ we have 
\[
0=P(\phi)-\frac{1}{4}K(\phi)=\frac{3}{4}\norm{\phi'}^2_{L^2}- \frac{1}{12}\norm{\phi}^3_{L^3}-\frac{1}{20}\norm{\phi}^5_{L^5}.
\]
Thus,
\[
\norm{\phi'}^2_{L^2}=\frac{1}{9}\norm{\phi}^3_{L^3}+\frac{1}{15}\norm{\phi}^5_{L^5}.
\]
It follows that
\begin{align}
\partial_{\lambda}^2S(\phi^{\lambda})\vert_{\lambda=1} &=\frac{1}{36}\norm{\phi}^3_{L^3}-\frac{1}{12}\norm{\phi}^5_{L^5}\nonumber\\
&= \frac{1}{36}\norm{\phi}^3_{L^3}-\frac{1}{12}\frac{10}{3}\left(\norm{\phi'}^2_{L^2}+\frac{1}{6}\norm{\phi}^3_{L^3}-\frac{a_2}{4}\norm{\phi}^4_{L^4}-P(\phi)\right)\nonumber\\
&= -\frac{5}{18}\norm{\phi'}^2_{L^2}-\frac{1}{54}\norm{\phi}^3_{L^3}+\frac{5a_2}{72}\norm{\phi}^4_{L^4}. \label{equality new}
\end{align}
Thus, 
\[
\partial_{\lambda}^2S(\phi^{\lambda})\vert_{\lambda=1} <0.
\]  
This implies the instability of algebraic standing waves in the case DDF. \\
In the case DFF, using \eqref{equality new} and the fact that $a \norm{\phi}^3_{L^3}+b\norm{\phi}^5_{L^5}\geq 2\sqrt{ab}\norm{\phi}^4_{L^4}$ for all $a,b>0$ we have 
\begin{align*}
\partial_{\lambda}^2S(\phi^{\lambda})\vert_{\lambda=1}&=-\frac{5}{18}\left(\frac{1}{9}\norm{\phi}^3_{L^3}+\frac{1}{15}\norm{\phi}^5_{L^5}\right)-\frac{1}{54}\norm{\phi}^3_{L^3}+\frac{5a_2}{72}\norm{\phi}^4_{L^4}\\
&= -\frac{4}{81}\norm{\phi}^3_{L^3}-\frac{1}{54}\norm{\phi}^5_{L^5}+\frac{5a_2}{72}\norm{\phi}^4_{L^4}\\
&\leq  -\frac{4}{27 \sqrt{6}} \norm{\phi}^4_{L^4}+\frac{5a_2}{72}\norm{\phi}^4_{L^4}<0,
\end{align*}
since we have assumed $a_2 < \frac{32}{15\sqrt{6}}$. Thus, in the case DFF and $a_2<\frac{32}{15\sqrt{6}}$ we obtain the instability of algebraic standing waves. \\
Case $n=2$:\\
In this case, we have
\begin{equation}\label{eqqq}
\partial^2_{\lambda}S(\phi^{\lambda})\vert_{\lambda=1}=\norm{\nabla\phi}^2_{L^2}-\frac{a_2}{2}\norm{\phi}^4_{L^4}-\frac{6}{5}\norm{\phi}^5_{L^5}.
\end{equation}
Moreover,
\[
0=P(\phi)=\norm{\nabla \phi}^2_{L^2}+\frac{1}{3}\norm{\phi}^3_{L^3}-\frac{a_2}{2}\norm{\phi}^4_{L^4}-\frac{3}{5}\norm{\phi}^5_{L^5}.
\]
Replacing $\frac{a_2}{2}\norm{\phi}^4_{L^4}=\norm{\nabla \phi}^2_{L^2}+\frac{1}{3}\norm{\phi}^3_{L^3}-\frac{3}{5}\norm{\phi}^5_{L^5}$ in \eqref{eqqq}, we obtain
\[
\partial^2_{\lambda}S(\phi^{\lambda})\vert_{\lambda=1}=-\frac{1}{3}\norm{\phi}^3_{L^3}-\frac{3}{5}\norm{\phi}^5_{L^5}<0.
\]
The instability of algebraic standing waves in the case $n=2$ follows.\\
Case $n=3$:\\
In this case, we have
\begin{equation}\label{eqqq2}
\partial^2_{\lambda}S(\phi^{\lambda})\vert_{\lambda=1}=\norm{\nabla\phi}^2_{L^2}+\frac{1}{4}\norm{\phi}^3_{L^3}-\frac{3 a_2}{2}\norm{\phi}^4_{L^4}-\frac{63}{20}\norm{\phi}^5_{L^5}.
\end{equation}
Moreover,
\[
0=P(\phi)=\norm{\nabla \phi}^2_{L^2}+\frac{1}{2}\norm{\phi}^3_{L^3}-\frac{3 a_2}{4}\norm{\phi}^4_{L^4}-\frac{9}{10}\norm{\phi}^5_{L^5}.
\]
Hence, 
\begin{align*}
\partial^2_{\lambda}S(\phi^{\lambda})\vert_{\lambda=1}&=\partial^2_{\lambda}S(\phi^{\lambda})\vert_{\lambda=1}-2P(\phi)\\
&=-\norm{\nabla \phi}^2_{L^2}-\frac{3}{4}\norm{\phi}^3_{L^3}-\frac{27}{20}\norm{\phi}^5_{L^5}<0.
\end{align*}
The instability of algebraic standing waves in case $n=3$ follows. This completes the proof of Theorem \ref{Theorem 1}.
\end{proof}

\begin{remark}
We assume that the assumption of Theorem \ref{Theorem 1} holds. Let $e^{i\omega t}\phi_{\omega}(x)$ be standing wave of \eqref{triple power equation} with $\phi_{\omega}$ is a radial positive solution of \eqref{SW profile}. By similar as Proposition \ref{pro instable 1}, if
\begin{equation}\label{00}
\partial_{\lambda}^2S_{\omega}(\phi_{\omega}^{\lambda})\vert_{\lambda=1}<0
\end{equation}
then the standing wave $e^{i\omega t}\phi_{\omega}(x)$ is orbitally unstable. By continuity of the maps $\omega \mapsto \phi_{\omega}$ and $\omega\mapsto S_{\omega}$ on $\R^+$, we see that for $\omega>0$ small enough, the condition \eqref{00} holds. This implies that the standing waves are orbitally unstable in the case of small frequency. This goes back to the result of \cite{LiTsIa21} in dimension $n=1$ for the case D*F (see Figure 3 \cite{LiTsIa21}).   
\end{remark}
When $n=2,3$, the highest power of \eqref{triple power equation} is $L^2$-supercritical. We use similar argument in \cite{FuHa21} to prove Theorem \ref{theorem 2}. From now on, we assume that the assumptions in Theorem \ref{theorem 2} hold. Let $S_{\omega}$, $K_{\omega}$ be defined as in \eqref{S omega}, \eqref{K omega} respectively. We define
\begin{align*}
P(v)&:=\norm{\nabla v}^2_{L^2}+\frac{n}{6}\norm{v}^3_{L^3}-\frac{na_2}{4}\norm{v}^4_{L^4}-\frac{3n}{10}\norm{v}^5_{L^5},\\
\mu(\omega)&:=\inf\{S_{\omega}(v): v\in H^1(\R^n)\setminus \{0\}, K_{\omega}(v)=0 \},\\
\mathcal{B}_{\omega}&:= \{v\in H^1(\R^n): S_{\omega}(v)<\mu(\omega),P(v)<0\}.
\end{align*}
We can rewrite that
\[
P(v)=\partial_{\lambda}S_{\omega}(v^{\lambda})\vert_{\lambda=1},\quad \text{ where } v^{\lambda}(x):=\lambda^{\frac{n}{2}}v(\lambda x).
\]
Since the assumption of Theorem \ref{theorem 2}, we have $S_{\omega}(\phi_{\omega})=\mu(\omega)$. Let $u_{0} \in \Sigma$, where $\Sigma$ is defined by \eqref{defineSigma} and $u$ be the associated solution of \eqref{triple power equation} with the initial data $u(t=0)=u_0$. We see that $u \in C(I,\Sigma)$, where $0\in I$ is the maximal existence interval of $u$. Moreover, we have the following identity
\begin{equation}
\label{virial identity}
\frac{d^2}{dt^2}\norm{xu(t)}^2_{L^2}=8P(u(t)),
\end{equation} 
for all $t\in I$. We have the following result. 
\begin{lemma}\label{lm new}
If $v\in H^1(\R^n)$, $v \neq 0$ and $P(v)\leq 0$ then
\[
\frac{1}{2}P(v)\leq S_{\omega}(v)-\mu(\omega).
\]
\end{lemma}
\begin{proof}
Define
\begin{align*}
f(\lambda)&:=S_{\omega}(v^{\lambda})-\frac{\lambda^2}{2}P(v)\\
&=\frac{\omega}{2}\norm{v}^2_{L^2}+\frac{\lambda^2}{2}(\norm{\nabla v}^2_{L^2}-P(v))+\frac{1}{3}\lambda^{\frac{n}{2}}\norm{v}^3_{L^3}-\frac{a_2}{4}\lambda^n\norm{v}^4_{L^4}-\frac{1}{5}\lambda^{\frac{3n}{2}}\norm{v}^5_{L^5}.
\end{align*}
From the definition of $P$, we have $f'(1)=0$. Moreover, we have
\begin{align*}
f'(\lambda)&=\lambda(\norm{\nabla v}^2_{L^2}-P(v))+\frac{n}{6}\lambda^{\frac{n}{2}-1}\norm{v}^3_{L^3}-\frac{n a_2}{4}\lambda^{n-1}\norm{v}^4_{L^4}-\frac{3n}{10}\lambda^{\frac{3n}{2}-1}\norm{v}^5_{L^5}\\
&=-\lambda\left(-\frac{n}{6}\norm{v}^3_{L^3}-\frac{n a_2}{4}\norm{v}^4_{L^4}-\frac{3n}{10}\norm{v}^5_{L^5}\right)+\frac{n}{6}\lambda^{\frac{n}{2}-1}\norm{v}^3_{L^3}-\frac{n a_2}{4}\lambda^{n-1}\norm{v}^4_{L^4}-\frac{3n}{10}\lambda^{\frac{3n}{2}-1}\norm{v}^5_{L^5}\\
&=-\lambda \left(\frac{n}{6}\norm{v}^3_{L^3}(1-\lambda^{\frac{n}{2}-2})+\frac{na_2}{4}\norm{v}^4_{L^4}(\lambda^{n-2}-1)+\frac{3n}{10}\norm{v}^5_{L^5}(\lambda^{\frac{3n}{2}-2}-1)\right).
\end{align*}
When $n=2$, $f'(\lambda)=\lambda(1-\lambda)\left(\frac{1}{3\lambda}\norm{v}^3_{L^3}+\frac{3}{5}\norm{v}^5_{L^5}\right)$. Thus, $f'$ is positive if $\lambda<1$ and negative if $\lambda>1$. When $n=3$, 
\[
f'(\lambda)=\lambda(1-\lambda^{\frac{1}{2}})\left(\frac{1}{2}\norm{v}^3_{L^3}\lambda^{\frac{-1}{2}}+\frac{3a_2}{2}\norm{v}^4_{L^4}(\lambda^{\frac{1}{2}}+1)+\frac{9}{10}\norm{v}^5_{L^5}\frac{\lambda^2+\lambda^{\frac{3}{2}}+...+1}{\lambda^{\frac{1}{2}+1}}\right).
\] 
Since $a_2>-\varepsilon$ for $\varepsilon>0$ small enough, we may prove that the sign of $f'$ is similar as in the case $n=2$ (we use Cauchy inequality and the fact that $\norm{v}^8_{L^4}\leq \norm{v}^3_{L^3}\norm{v}^5_{L^5}$ and $(\lambda^{\frac{1}{2}}+1)^4\lesssim\lambda^2+\lambda^{\frac{3}{2}}+...+1$). Thus, we have
\[
f(1)=\max\{f(\lambda): \lambda>0\}.
\]
Since $n \geq 2$, we have $K_{\omega}(v^{\lambda})>0$ if $\lambda=0$ and $K_{\omega}(v^{\lambda})<0$ if $\lambda$ large enough. Thus, there exists $\lambda_0>0$ such that $K_{\omega}(v^{\lambda_0})=0$. By the definition of $\mu(\lambda)$ and $P(v)\leq 0$, we have
\[
\mu(\lambda)<S_{\omega}{v^{\lambda_0}}<S_{\omega}{v^{\lambda_0}}-\frac{\lambda_0^2}{2}P(v)=f(\lambda_0)\leq f(1)=S_{\omega}(v)-\frac{1}{2}P(v).
\]
This completes the proof of Lemma \ref{lm new}.
\end{proof}
\begin{lemma}\label{lm4.4}
The set $\mathcal{B}_{\omega}$ is invariant under flow of \eqref{triple power equation}.
\end{lemma}
\begin{proof}
Let $u_0\in \mathcal{B}_{\omega}$ and $u(t)$ be the associated solution of \eqref{triple power equation}. Since $S_{\omega}$ is conserved under flow of \eqref{triple power equation}, we have $S_{\omega}(u(t))=S_{\omega}(u_0)<\mu(\omega)$. Now, we show that $P(u(t))<0$. Assume that, this does nor hold. Then by continuity of the map $t \mapsto P(u(t))$, there exists $t_0>0$ such that $P(u(t_0))=0$. Using Lemma \ref{lm new}, we have $S_{\omega}(u(t_0)) \geq \mu(\omega)$. This contradict to the fact that $S_{\omega}(u(t))<\mu(\omega)$ for all $t$. This completes the proof.  
\end{proof}
We have the following result.
\begin{proposition}\label{pro4.5}
If $u_0\in \mathcal{B}_{\omega} \cap \Sigma$ then the associated solution $u(t)$ of \eqref{triple power equation} blows up in finite time. 
\end{proposition}
\begin{proof}
By Lemma \ref{lm4.4}, we have $u(t) \in \mathcal{B}_{\omega} \cap \Sigma$, for all $t \in I$ the maximal existence time of the solution. From the virial identity and Lemma \ref{lm4.4}, we have
\begin{align*}
\frac{d^2}{dt^2}\norm{xu(t)}^2_{L^2} &= 8P(u(t)) \leq 16 (S_{\omega}(u(t))-\mu(\omega))=16(S_{\omega}(u_0)-\mu(\omega))<0.
\end{align*} 
for all $t\in I$, which implies $|I|<\infty$. We obtain the desired result.
\end{proof}
\begin{proof}[proof of Theorem \ref{theorem 2}]
Using Proposition \ref{pro4.5}, it is sufficient to construct a sequence $varphi_n \in \mathcal{B}_{\omega}\cap \Sigma$ such that $\varphi_n\rightarrow \phi_{\omega}$ as $n\rightarrow\infty$. In our case, $a_1,a_2,a_3$ satisfy the assumption of Theorem \ref{Theorem 1}. Thus, by using similar argument in the proof of Theorem \ref{Theorem 1}, we may check that $\partial_{\lambda}^2S_{\omega}(\phi_{\omega}^{\lambda})\vert_{\lambda=1}<0$. Then, by using similar argument in the proof of Lemma \ref{lm chosing sequence}, we may pick $\varphi_n=\phi_{\omega}^{\lambda_n}$, for $\lambda_n>1$ such that $\lambda_n \rightarrow 1$ as $n\rightarrow\infty$. This completes the proof.
\end{proof}

\section*{Acknowledgement}

I wishe to thank Prof. Stefan Le Coz for the useful discussion and his encouragement. This work is supported by the ANR LabEx CIMI (grant ANR-11-LABX-0040) within the French State Programme “Investissements d’Avenir.

\bibliographystyle{abbrv}
\bibliography{bibliothefornam}

\end{document}